\documentclass[11pt]{amsart}

\usepackage{color}

\usepackage{tikz-cd}
\usepackage[all]{xy}
\usepackage[utf8]{inputenc}
\usepackage[english]{babel}
\usepackage[T1]{fontenc}
 \usepackage{lmodern}
\usepackage{mathrsfs}

\usepackage{graphicx,epsfig,amscd,graphics,xypic}
 \usepackage[all]{xy}
\usepackage{amsmath,amssymb}

\usepackage[margin=3cm]{geometry}

\usepackage{hyperref}
\hypersetup{
    bookmarks=true,         
    unicode=false,          
    pdftoolbar=true,        
    pdfmenubar=true,        
    pdffitwindow=false,     
    pdfstartview={FitH},    
    pdftitle={My title},    
    pdfauthor={Author},     
    pdfsubject={Subject},   
    pdfcreator={Creator},   
    pdfproducer={Producer}, 
    pdfkeywords={keyword1} {key2} {key3}, 
    pdfnewwindow=true,      
    colorlinks=false,       
    linkcolor=red,          
    citecolor=green,        
    filecolor=magenta,      
    urlcolor=cyan           
}

\newcommand{\A}{\mathcal{A}}

\newtheorem{theorem}{Theorem}
\newtheorem*{theorem*}{Theorem}
\newtheorem{corollary}{Corollary}
\newtheorem*{corollary*}{Corollary}
\newtheorem{lemma}{Lemma}
\newtheorem{proposition}{Proposition}

\newtheorem{conjecture}{Conjecture} 
\newtheorem{question}{Question} 
 
\newtheorem{remark}{Remark} 
\theoremstyle{definition}
\newtheorem{definition}{Definition}

\begin{document}

\title{The symplectic structure for renormalisation of circle diffeomorphisms with breaks}

\author[Selim Ghazouani]{S. Ghazouani}
\address{Mathematics Institute, University of Warwick, Coventry CV4 7AL, United Kingdom}
\email{s.ghazouani@warwick.ac.uk}

\author[Kostantin Khanin]{K. Khanin}
\address{3359 Mississauga Road,
Mississauga, ON L5L 1C6, Canada}
\email{khanin@math.utoronto.ca}

\maketitle

\begin{abstract}

In this article we prove that iterated renormalisations of $\mathcal{C}^r$ circle diffeomorphisms with $d$ breaks, $r>2$, with given size of breaks, converge to an invariant family of piecewise Moebius maps, of dimension $2d$. We prove that this invariant family identifies with a \textit{relative character variety} $\chi(\pi_1 \Sigma, \mathrm{PSL}(2,\mathbb{R}), \mathbf{h})$ where $\Sigma$ is a $d$-holed torus, and that the renormalisation operator identifies with a sub-action of the mapping class group $\mathrm{MCG}(\Sigma)$. This action is known to preserves a symplectic form, thanks to the work of Guruprasad-Huebschmann-Jeffrey-Weinstein \cite{GHJW}. Its pull-back through the aforementioned identification provides a symplectic form invariant by renormalisation.
\end{abstract}

\section{Introduction}

Renormalisation is a powerful method to analyse the dynamics of parabolic dynamical systems. It is of major importance in the study of many cases in one-dimensional dynamics, such as unimodal maps at the boundary of chaos \cite{CoulletTresser,McMullen,Sullivan}, Lorenz maps \cite{Winckler,MartensWinckler} or circle diffeomorphisms with singularities. For the latter case, renormalisation methods allowed to successfully understand rigidity phenomena in the case of circle diffeormorphisms \cite{Herman1,KhaninSinai,KhaninTeplinsky}, circle diffeomorphisms with one critical point \cite{deFariadeMelo,deFariadeMelo2} and circle diffeomorphisms with one break \cite{KhaninKhmelev,KhaninKocic,KhaninYampolsky,KhaninKocicMazzeo}. 

\vspace{2mm}

\paragraph*{\bf Circle diffeomorphisms with breaks} We will discuss in this article the renormalisation theory of circle diffeomorphisms with several break points. A circle diffeomorphism with breaks (or break points) is a circle homeomorphism smooth away from finitely many points (the \textit{breaks}) where the derivative has a discontinuity. The \textit{size} of a break is the ratio between the right and the left derivative. The case of a single break was extensively studied in the series of articles \cite{KhaninKhmelev} \cite{KhaninTeplinsky2}, \cite{KhaninKocic}, \cite{KhaninYampolsky} and \cite{KhaninKocicMazzeo}.  The major achievement is the following rigidity theorem

\begin{theorem}[Khanin-Kocic-Mazzeo, \cite{KhaninKocicMazzeo}]
There is a full measure set of numbers $\rho \in [0,1[$ such that the following holds. If $T_1$ and $T_2$ are two circle diffeomorphisms with one break, with same size of the break and same rotation number $\rho$, then $T_1$ and $T_2$ are $\mathcal{C}^1$-conjugate.
\end{theorem}

In the one-dimensional case, renormalisation can be reduced to the study of the renormalisation operator $\mathcal{R}$ restricted to a two-dimensional parameter space, which allows for an efficient explicit analysis. A difficulty one encounters when trying to extend these results to an arbitrary number of break points is that the dimension of the attractor of renormalisation increases, rendering an explicit description more arduous.

\vspace{2mm}

\paragraph*{\bf Renormalisation and geometry of surfaces} An invertible one-dimensional dynamical system can always be turned into a singular flow on a topological surface, by means of a suspension. This topological remark has in certain cases allowed for fruitful connections between the geometry of surfaces and the renormalisation of associated dynamical systems to be made. A good example of this is the connection between the Gauss map, which is the renormalisation operator for rigid rotations, and the geodesic flow on the modular surface  $\mathbb{H}/ \mathrm{PSL}(2, \mathbb{Z})$, which is the moduli space of elliptic curves/flat tori. This can be generalised to \textit{interval exchange transformations} together with the Rauzy-Veech induction, which is formally equivalent to the \textit{Teichmüller flow} on the moduli space of Abelian differentials on Riemann surfaces. 

Note that these cases are linear dynamical systems and the renormalisation theory is concerned with their ergodic and combinatorial properties. We are going to introduce a similar correspondence in a case which is essentially non-linear, via moduli spaces of representations of surface groups. In our case, renormalisation tells us about \textit{geometric} properties of the associated dynamical systems.

\vspace{2mm}

\paragraph*{\bf Character varieties and renormalisation} 

Given a topological surface $\Sigma$ and a reductive Lie group $G$, one can construct the associated character variety $$\chi(\pi_1 \Sigma, G) $$ which is roughly the space of representations of $\pi_1 \Sigma$ in $G$ up to the action of $G$ by conjugation. These character varieties naturally identify with the space of flat principal $G$-bundle over $\Sigma$. Attiyah and Bott showed in \cite{AtiyahBott} the existence of a natural symplectic form on $\chi(\pi_1 \Sigma, G) $  when $\Sigma$ is closed and orientable. In the particular case $G= \mathrm{PSL}(2, \mathbb{R})$,  this symplectic form was shown by Goldman in \cite{Goldman} to agree with the Weil-Petersson metric on Teichmüller space (which itself identifies with a connected component of $\chi(\pi_1 \Sigma, \mathrm{PSL}(2, \mathbb{R}) ) $). This symplectic form is defined in purely topological term and is therefore invariant by the action of $\mathrm{MCG}(\Sigma)$ the modular group of $\Sigma$. 

In this article we put forward a conceptual approach to the renormalisation of circle diffeomorphisms with $d$ breaks, based on a connection with representations of the fundamental group of a $d$-holed torus into $\mathrm{PSL}(2,\mathbb{R})$. The above material developed in the case of closed surfaces can be extended to surfaces with boundary, and for the theory to unfold nicely one has to work with \textit{relative character varieties}, which are conjugacy classes of representations whose conjugacy classes are fixed for images of boundary components. Our main theorem is

\begin{theorem*}

For all $d \geq 1$, there exists a $2d$ dimensional submanifold of the space of Moebius circle diffeomorphisms with $d$ breaks and fixed size of breaks $\mathbf{c}$ such that, for all $T$ with irrational rotation number and breaks of sizes $\mathbf{c}$, the following holds:

\begin{enumerate}

\item iterated renormalisations of $T$ converge exponentially fast to $\mathcal{E}(\mathbf{c})$;

\item connected components of $\mathcal{E}(\mathbf{c})$ identify with  open sets of the relative character variety $\chi(\Sigma_{1,d}, \mathrm{PSL}(2,\mathbb{R}), \mathbf{h})$ where $\mathbf{h}$ is entirely determined by $\mathbf{c}$;

\item the action of the renormalisation operator $\mathcal{R}$ on $\mathcal{E}(\mathbf{c})$ identifies with a sub-action of the mapping class group $\mathrm{MCG}(\Sigma_{1,d})$.

\end{enumerate}

\end{theorem*}

\noindent An important corollary of this theorem, implied by fundamental work of Guruprasad-Huebschmann-Jeffrey-Weinstein \cite{GHJW} on relative character varieties, is the following

\begin{corollary*}
The renormalisation operator $\mathcal{R}: \mathcal{E}(\mathbf{c}) \longrightarrow \mathcal{E}(\mathbf{c})$ preserves a symplectic form.
\end{corollary*}

\section{Circle diffeormorphisms with breaks}

\noindent In this section we collect standard material on circle diffeomorphisms with breaks.

\begin{definition}

A circle $\mathcal{C}^r$-diffeomorphism with breaks is a circle homeomorphism $T$ satisfying the following two conditions:

\begin{itemize}

\item there are finitely many points $p_1 =0 < \cdots < p_d < 0 = p_1 \in S^1$ such that $T$ is $\mathcal{C}^r$ away from these points;

\item $T$ extends to a $\mathcal{C}^r$-diffeomorphism on every segment of the form $[p_i, p_{i+1}]$; 

\end{itemize}

\noindent The points $(p_i)$ are called the \textit{break points} of $T$. For each $p_i$, the ratio $\frac{T'(p_i^+)}{T'(p_i^-)}$ is called the \textit{size of the break}. Here, $T'(p_i^+)$ and $T'(p_i^-)$ respectively denote the derivative of $T$ at $p_i$ on the right and on the left.

\end{definition}

\noindent In this article, \textbf{we make the standing assumption that} $r > 2$.

\vspace{3mm}

\paragraph{\bf Rotation number and decorated rotation number}

For $T$ a circle homeomorphism we denote by $\rho(T) \in S^1$ its rotation number. We have the following theorem due to Denjoy 

\begin{theorem}[Denjoy]

Let $T$ be a circle homeomorphism which is differentiable away from finitely many points and assume that its derivative has bounded variation. Assume further that $\rho(T)$ is irrational. Then $T$ is minimal and is topologically conjugate to the rotation of angle $\rho(T)$. Furthermore, the conjugating map is unique up to composition by a rigid rotation.

\end{theorem}  

\noindent Denjoy's theorem in particular applies to circle diffeomorphisms with breaks. If $T$ is a diffeomorphism with breaks with irrational rotation number, there is a unique map conjugating it to its linear model which maps $p_1$ to $0$. The \textit{decorated rotation number} of $T$ is in that case the data of $\rho$ together with the images in $S^1$ by the conjugating map of $p_2, \cdots,  p_{d-1}$ and $p_d$. 

\vspace{3mm}

\paragraph{\bf Continued fraction algorithm}

\noindent The \textit{continued fraction algorithm} is a procedure which allows to find the best rational approximations of a given irrational number. We follow here \cite{Herman1}.  Let $G := x \mapsto \frac{1}{x}$ if $x \neq 0$ and let $G(0) = 0$, $G$ is called the \textit{Gauss map}. Define $a(x) = [\frac{1}{x}]$ where $[ \cdot ]$ denotes the integral part. Given $\alpha \in [0,1$, we define the sequence $a_1, \cdots, a_n, \cdots$ by

$$ a_i = a(G^i(\alpha)).$$ We use the notation 

$$ [a_1, \cdots, a_n] = \frac{p_n}{q_n} = \frac{1}{a_1 + \frac{1}{a_2 + \frac{1}{\cdots + \frac{1}{a_n}}}} $$ The numbers $\frac{p_n}{q_n}$ are good rational approximations of $\alpha$ in the sense that they satisfy $| \alpha - \frac{p_n}{q_n} | < \frac{1}{q_n^2}$. Diophantine approximation is a topic in its own right and we do not intend to say anything more about it. We just want to insist upon the fact that given a circle homeomorphism $T$ of rotation number $\alpha \notin \mathbb{Q}$, the iterations $T^{q_n}$ will play an important in the theory, as it is illustrated by the next result.

\vspace{3mm}

\paragraph{\bf The Denjoy-Koksma inequality}

The Denjoy-Koksma inequality is a remarkable result on the ergodic theory of circle homeomorphism. Recall that any circle homeomorphism of irrational rotation number is uniquely ergodic.

\begin{theorem}[Denjoy-Koksma inequality]
\label{DK}
Let $\alpha$ be an irrational rotation number. Let $T$ be a circle homeomorphism of rotation number $\alpha$ and let $\mu$ be the unique invariant measure of $T$. Let $f : S^1 \longrightarrow \mathbb{R}$ be a measurable function of bounded variation. Assume further that $\int{fd\mu} = 0$. Then, for all $n \in \mathbb{N}^*$ and for all $x \in S^1$

$$  | \sum_{i=0}^{q_n-1}{f \circ T^k(x)} | \leq \mathrm{Var}(f) $$ where $\mathrm{Var}(f)$ is the total variation of $f$.
\end{theorem}

\noindent The interested reader will find the proof of this statement in \cite{Herman1}, Chapitre VI.3.

\vspace{3mm}

\paragraph{\bf Distortion bounds} This paragraph is dedicated to proving the following result.

\begin{lemma}
\label{bound1}
Let $T$ be a circle diffeomorphim with breaks Let $J \subset [0,1]$ be an interval such that $J, T(J), T^2(J), \cdots, T^n(J)$ are pairwise disjoint and do not contain break points of $T$. Then for all $x,y \in J$ we have 

$$ \frac{\mathrm{D}(T^n)(x)}{\mathrm{D}(T^n)(y)} \leq \exp(\int_0^1{|\mathrm{D} \log \mathrm{D}T|\mathrm{dLeb}}). $$

\end{lemma}

\begin{proof}

The proof is classical. We have that $$ \log \mathrm{D}T^n(x) = \sum_{i=0}^{n-1}{\log \mathrm{D}T (T^i(x) )} $$ and therefore 

$$ | \log \mathrm{D}T^n(x) - \log \mathrm{D}T^n(y) | \leq  \sum_{i=0}^{n-1}{|\log \mathrm{D}T (T^i(x) ) -  \log \mathrm{D}T (T^i(x) ) |}  \leq  \sum_{i=0}^{n-1}{|\int_{T^i(y)}^{T^i(x)}{\eta_{T}}|}.$$  Since the intervals $[T^i(y), T^i(x)]$ are pairwise disjoints we get 

$$ | \log \mathrm{D}T^n(x) - \log \mathrm{D}T^n(y) | \leq \int_0^1{|\mathrm{D} \log \mathrm{D}T|\mathrm{dLeb}} $$ and exponentiating gives the expected result.

\end{proof}

\section{Rauzy induction for generalised interval exchange transformations}

\noindent The formalism of \textit{generalised interval transformations} is convenient to define the renormalisation scheme we will need for circle diffeomorphisms with breaks. In this section we introduce both this formalism and the renormalisation scheme, and explain how these relate to circle diffeomorphisms with breaks.

\subsection{Generalised interval exchange transformations and Rauzy induction}

In this subsection, we introduce some notation. 

\begin{itemize}

\item  $\mathcal{A}$ is an alphabet on $d$ letters/symbols.  

\item A \textit{marked permutation} is the datum of two bijections $\pi = (\pi_t, \pi_b)$ between the sets $\{1, \ldots, d \}$ and $\mathcal{A}$. Here $t$ stands for \textit{top} and $b$ stands for \textit{bottom}.

\item Let $0 < u_1^t < \ldots < u_{d-1}^t < 1$ and   $0 < u_1^b < \ldots < u_{d-1}^b < 1$. For all $i \in \{1, \ldots, d \}$, denote by $I_{\pi_t(i)}$ the interval $[u_{i-1}^t, u_i^t[$ and by  $I_{\pi_b(i)}$ the interval $[u_{i-1}^b, u_i^b[$.

\end{itemize}

\begin{definition}

A $\mathcal{C}^r$-\textit{generalised interval exchange transformation} (GIET) is a bijection of $[0,1[$ which satisfies the following conditions:
 
 \begin{enumerate}
 
\item  for any $\alpha \in \A$, it maps $I_{\alpha}^t$ onto $I_{\alpha}^b$;

\item for any $\alpha \in \A$, it is $\mathcal{C}^{r}$ diffeomorphism between,  $I_{\alpha}^t$ and $I_{\alpha}^b$ ;

\item for any $\alpha \in \A$, it extends to the closure of $I_{\alpha}^t$ , is derivable at the end points of $I_{\alpha}^t$ and the derivative is positive. 
 
 \end{enumerate}

\noindent The points $0 < u_1^t < \ldots < u_{d-1}^t < 1$ are called the \textit{singularities} of the GIET. 

\end{definition}

\noindent Note that there is a natural way to associate a GIET to $f$ a circle diffeomorphism with breaks. By arbitrarily choosing one of the break $b$, one can "cut" the circle  at this break to get an identification of $S^1 \setminus \{b\}$ with $]0,1[$. The map induced by $f$ on $]0,1[$ extends to a GIET whose singularities are exactly the images of the break points under this identifications.

\noindent We define hereafter an algorithm on GIETs which will provide us with a renormalisation scheme for GIETs and circle diffeomorphisms with breaks in particular.

\vspace{3mm}

\paragraph*{\bf Elementary step of Rauzy induction}

The Rauzy induction consists in taking a first return map on a well-chosen subinterval. Consider $T$ a GIET. One should think of the points $u_{\pi(d-1)}^t$ and $u_{\pi(d-1)}^b$ as the rightmost singularities at the top and the bottom respectively. 

\noindent The \textit{elementary step} of Rauzy induction consists in taking the first return map of $T$ on $[0,u_{\pi(d-1)}^t[$ if $u_{\pi(d-1)}^t > u_{\pi(d-1)}^b$ or on $[0,u_{\pi(d-1)}^b[$ if $u_{\pi(d-1)}^b > u_{\pi(d-1)}^t$. After affinely rescaling to $[0,1]$, this operation returns $E(T)$ a new GIET with a different marked permutation $\pi'$. We make the following observation: $\pi'$ only depends on $\pi$ and on whether $u_{\pi(d-1)}^t > u_{\pi(d-1)}^b$ or $u_{\pi(d-1)}^b > u_{\pi(d-1)}^t$. These two cases will be referred to as "top wins" and "bottom wins" respectively. In each case, the letter corresponding to the longest interval ( $\pi_t(d-1)$ if top wins and $\pi_b(d-1)$ if bottom wins) is called the \textit{winner}.

\vspace{3mm}

\subsection{Rauzy diagram, Rauzy classes and Rauzy paths}

A \textit{Rauzy diagram} is a graph whose set of vertices are marked permutations on $d$ intervals and oriented arrows between two vertices $\pi$ and $\pi'$ if $\pi'$ is obtained from $\pi'$ after an elementary step of Rauzy induction. 
\noindent A \textit{Rauzy class} is a connected component of a Rauzy diagram. One will find examples and detailed discussions in \cite{Yoccoz2}.

\noindent A generalised interval exchange transformation $T$ is said to be \textit{infinitely renormalisable} if $E^n(T)$ is well-defined for all $n \geq 0$. Such an infinitely renormalisable GIET defines an admissible path in the Rauzy diagram called its \textit{Rauzy path} and which we will denote by $\gamma(T)$. Finally, a Rauzy path $\gamma$ is called $\infty$\textit{-complete} if every letter of $\mathcal{A}$ is a winner at a step of the induction infinitely many times along $\gamma$ (which does not necessarily mean that $\gamma$ visits all vertices of the Rauzy classes it lives in).

\subsection{Circular Rauzy classes and circle diffeomorphisms with breaks.}

As sketched above, a circle diffeomorphism with $d-1$ breaks can be turned into a generalised IET the following way. Choose a break $p_0 \in S^1$ and identify $S^1 \setminus \{ p_0 \} $ with $]0,1[$. The induced map extends to $[0,1]$ to a GIET on $d$ intervals. Note that there are different ways to perform this operation, all corresponding to the choice of a break point. At any rate, the associated permutation $\pi$ will be circular. However, all permutations contained in the Rauzy class of a circular permutation are not circular. We have the following Proposition:

\begin{proposition}

The following statements hold true.
\begin{enumerate}

\item A GIET is infinitely renormalisable if and only if its singular points lies on different orbits.

\item A GIET induced by a circle diffeomorphism with breaks has $\infty$-complete Rauzy path if and only if its rotation number is irrational.

\item Two minimal circle diffeomorphisms with breaks have same decorated rotation number if and only if their induced GIETs have same Rauzy paths.
\end{enumerate}
\end{proposition}

\noindent For the proof of this rather elementary result, we refer to \cite{Yoccoz3}.

\vspace{2mm}

\paragraph{\bf Projection to $2$-GIETs and acceleration of the induction} Another important remark at this point is that a $\mathcal{C}^r$-GIET with $d$ intervals whose permutation is circular induces a $\mathcal{C}^0$-GIET on two intervals (by grouping together adjacent intervals mapped to adjacent intervals). Rauzy induction is also well-defined for such $2$-GIETs and defines an algorithm which is formally different from the Rauzy induction on the same GIET thought as a $d$-GIET. These two are nonetheless related the following way: the algorithm on two intervals is an \textit{acceleration} of the algorithm on $d$ intervals. Formally this means that for any renormalisable GIET there exists an integer $k$ such that applying $k$ steps of the Rauzy induction on $d$ intervals produces the same outcome as one application of the Rauzy induction on two  intervals. It is important for us to make this distinction as we want to keep track of the combinatorial structure of the orbits of the breaks points.

\vspace{2mm}

\paragraph{\bf Standard renormalisation for circle homeomorphisms}

As just explained above, circle homeomorphism can be turned into $2$-GIETs. There exists several equivalent definitions of renormalisation for such maps. We explain here the one which uses Rauzy induction. 
\noindent The definition consists in grouping elementary steps of the Rauzy induction the following way.

\begin{itemize}

\item If top is the winner, apply $E$ as long as top keeps winning. When bottom eventually wins, apply $E$, apply $E$ once more. 

\item If bottom is the winner, apply $E$ as long as bottom keeps winning. When top eventually wins, apply $E$, apply $E$ once more. 

\end{itemize}

\noindent Note that it can be the case that this algorithm gets stuck in a loop and does not terminate. When the algorithm terminates, we can define $\mathcal{R}$ the standard renormalisation for circle homeomorphism. We have the following Proposition

\begin{proposition}

Let $T$ be a $2$-GIET corresponding to a circle homeomorphism of irrational rotation number. Then $\mathcal{R}^nT$ is well-defined for all $n \geq 0$. 
\end{proposition}

\noindent This is a standard fact about circle homeomorphisms and generalised interval exchange transformations. The interested reader will find proofs and further discussions in \cite{Yoccoz3}.

\noindent We now explain how renormalisation relates to the continued fraction expansion of $\rho(T)$. Assume $T$ is a $2$-GIET such that $\rho(T)$ is irrational. Denote by $k_n \geq 2$ the integer such that $\mathcal{R}^n(\mathcal{R}^{n-1}T) = E^{k_n}(\mathcal{R}^{n-1}T)$. Let $(a_n)_{n\in\mathbb{N}^*}$ the sequence of positive integers such that 

$$ \rho(T) = \frac{1}{1 + \frac{a_1}{1+\frac{a_2}{1 + \cdots}}}.$$ We then we have  $$ \forall n \geq 1, \ a_n = k_n -1.$$

\begin{remark}

Here $\mathcal{R}$ is defined using renormalisations for $2$-GIETs. However, $\mathcal{R}$ is an operator acting on $d$-GIET. In particular, two $d$-GIET $T_1$ and $T_2$ may have different Rauzy paths when there associated $2$-GIETs have same Rauzy paths. It is the case when $T_1$ and $T_2$ having same rotation number  but their decorated rotation numbers differ.

\end{remark}

\subsection{Dynamical partitions}

Let $T$ be a GIET and assume further that $T$ is $n$ times renormalisable. For any $n \geq 0$, $\mathcal{R}^nT$ is the rescaling of a first return map of $T$ on an interval of the form $[0,x_n]$. The interval $[0,x_n]$ is partitioned into 

$$ [0,x_n] = \cup_{\alpha \in \mathcal{A}}^d{I^t_{\alpha}(n)} $$ and $\mathcal{R}^nT$ rescaled down to $[0,x_n]$ is equal to $T^{l^n_{\alpha}}$ on each of the $I^t_{\alpha}(n)$. For $\alpha \in \mathcal{A}$, we introduce

$$ \mathcal{P}^n_{\alpha} = \{ I^j_n, T(I^t_{\alpha}(n)),  T^2(I^t_{\alpha}(n)), \cdots, T^{l^n_{\alpha}-1}(I^t_{\alpha}(n))  \} $$ and we call 

$$ \mathcal{P}_n = \bigcup_{\alpha \in \mathcal{A}}{\mathcal{P}^n_{\alpha}} $$ the dynamical partition of level $n$. One easily verifies that $\mathcal{P}_n$ is a partition of $[0,1]$ into subintervals.

\section{Fast convergence of renormalisation towards Moëbius maps}

\noindent We fix $d \in \mathbb{N}^*$ and $\mathbf{c} \in \mathbb{R}_+^d$ a break profile. All constants will be implicitly assumed to only depend on $c$. In this section we prove the following theorem (which is a straightforward generalisation of a well-known fact for the case $d=1$ and should not be considered as new material). In the sequel $\mathbf{P}$ denotes the set of generalised IETs with circular permuations which are Moebius maps restricted to their branches.

\begin{theorem}
\label{convergenceMoebius}
There there exists $\rho < 1$ such that the following holds.
Let $T$ be circle diffeomorphism with $d-1$ break points whose rotation number is irrational. Then there exists $C_T >0$ such that  $$ d_1(\mathcal{R}^nT, \mathbf{P}) \leq C_T \rho^n.$$
\end{theorem}

\subsection{Size of the dynamical partition}

Let $T$ be a minimal (equivalently irrational rotation number) GIET with cyclic permutation. We denote by $\Delta_n(T) = \Delta_n = \sup_{J \in \mathcal{P}_n}{|J|}$. 

\begin{proposition}
\label{estimatessize}
There exists $\alpha<1$ such that for all $T$, there exists $D_T$ such that 

$$ \forall n \in \mathbb{N}, \  |\Delta_n | \leq D_T \cdot \alpha^n.$$

\end{proposition}

\begin{proof}

Note that the derivative of $T$ is a piecewise $\mathcal{C}^{r-1}$ function and therefore has bounded variation. We can think of $T$ as a $\mathcal{C}^0$ GIET on $2$-intervals, with derivative well-defined everywhere but at finitely many points, and with bounded variation. We can consider $\mathcal{P}'_n$ the dynamical partition associated with the renormalisation of this $2$-GIET. The dynamical partion $\mathcal{P}_n$ is just a refinement of $\mathcal{P}'_n$, and therefore if we denote by $\Delta'_n := \sup_{J \in \mathcal{P}'_n}{|J|}$ we have 

$$ \Delta_n \leq   \Delta'_n.$$ 

\noindent We give only a quick sketch for the rest of the proof. By Denjoy-Koksma inequality (Theorem \ref{DK}), the derivative of consecutive renormalisations remain uniformly bounded below and above. Using that fact, one can prove that the intervals at the base of the dynamical partition see their lengths divided by a uniform constant every two iterations of $\mathcal{R}$. The dynamical partition $\mathcal{P}'_{n+2}$ is the refining of $\mathcal{P}'_n$ obtained by propagating the partition of interval the of $\mathcal{R}^nT$. Because of the distortion bound of Proposition \ref{bound1} applied to the dynamical partition, the induced subdivision of each element of $\mathcal{P}_n$ is "balanced" and altogether this gives the existence of a constant $\kappa<1$ such that 

$$ \frac{\Delta_{n'+2}}{\Delta_n} \leq \kappa.$$

\noindent  The Proposition is implied by the inequality $\Delta_n \leq \Delta'_n$. A quantitative version of the above reasoning  can be found in \cite{KhaninSinai}, see Lemma 2 therein and its proof.
\end{proof}

\subsection{$\mathcal{C}^2$-bounds}

In this paragraph we prove an estimate which gives some uniform bounds on the second derivative of iterated renormalisations. The proof builds upon Lemma \ref{bound1}.

\begin{lemma}
\label{secondderivative}
Let $\varphi_1, \cdots, \varphi_n \in \mathcal{C}^2(\mathbb{R}, \mathbb{R})$. For all $k \leq n$ define $f_k = \varphi_k \circ \varphi_{k-1} \circ \cdots \circ \varphi_1$ and set $f_0 = \mathrm{Id}$. Then we have for all $n \geq 2$ the formula 

$$ f_n'' =  (f'_{n-1})^2 \cdot (\varphi''_{n} \circ f_{n-1}) +  \sum_{k=2}^n{ (f'_{n-k})^2 \cdot (\varphi''_{n-k+1} \circ f_{n-k})  \cdot (\varphi_n \circ \cdots \circ \varphi_{n-k+2})' \circ f_{n-k+1}   } $$

\end{lemma}

\begin{proof}

We proceed by induction on $n$. We check that the statement holds true for $n = 2$: 

$$ f_2'' = (\varphi_2 \circ \varphi_1)' = (\varphi'_1 \cdot \varphi_2'\circ \varphi_1)' = (\varphi'_1)^2 \cdot \varphi''_2 \circ \varphi_1 + \varphi''_1 \cdot \varphi_2' \circ \varphi_1.$$ 

\noindent Assume the statement holds true for $n \geq 2$. We have 

$$ f''_{n+1} = (\varphi_{n+1} \circ f_n)'' = \varphi''_{n+1} \circ f_n \cdot (f'_n)^2 + f''_n \cdot \varphi'_{n+1}\circ f_n.$$ Replacing $f''_n$ in the formula we get 

\begin{equation*}
\begin{split}
 & f''_{n+1} = (\varphi_{n+1} \circ f_n)'' =   \varphi''_{n+1} \circ f_n \cdot (f'_n)^2 +    (\varphi'_{n+1}\circ f_n) \cdot (f'_{n-1})^2 \cdot (\varphi''_{n} \circ f_{n-1}) \\
 &   + \sum_{k=2}^n{ (f'_{n-k})^2 \cdot (\varphi''_{n-k+1} \circ f_{n-k}) \cdot  (\varphi'_{n+1}\circ f_n) \cdot (\varphi_n \circ \cdots \circ \varphi_{n-k+2})' \circ f_{n-k+1}}. 
 \end{split}
\end{equation*} By the chain rule we have 

$$ (\varphi'_{n+1}\circ f_n) \cdot (\varphi_n \circ \cdots \circ \varphi_{n-k+2})' \circ f_{n-k+1}   =   ( \varphi_{n+1} \circ \varphi_n \circ \cdots \circ \varphi_{n-k+2})' \circ f_{n-k+1}$$ 

\noindent Injecting in the formula above for $f''_{n+1}$ gives the expected result.
\end{proof}

\noindent Consider a $\mathcal{C}^2$, increasing diffeomorphism $f : I \longrightarrow J$ where $I$ and $J$ are two connected intervals. We denote by $\mathrm{N}(f)$ the \textit{normalisation} or \textit{rescaling} of $f$, it is by definition the map $f$ pre-composed by the unique affine map sending $[0,1]$ onto $I$ and post-composed  by the unique affine map sending $J$ onto $[0,1]$. We have the following easy lemma:

\begin{lemma}
\label{normalised}
Let $f$ as above. Then we have 

$$ || \mathrm{N}(f)'' || \leq ||f'^{-1}|| \cdot ||f''|| \cdot |I| $$ 

\end{lemma}

\begin{proof}

\noindent Let $a = |I|$ and $b = |J|$. By definition we have 

$$ \mathrm{N}(f) := x \longmapsto \frac{1}{b} f(ax).$$ Thus 

$$ \mathrm{N}(f)''(x) = \frac{a^2}{b} f(ax) = a \frac{a}{b} f''(ax).$$ There exists $x_0 \in I$ such that  $\frac{1}{f'(x_0)} = \frac{|I|}{|J|} = \frac{a}{b}$. Hence the result.

\end{proof}

\noindent Using Lemma \ref{secondderivative} and Lemma \ref{normalised}, we prove the following 

\begin{proposition}
\label{controlC2}
Let $T$ be a $\mathcal{C}^2$-generalised IET and assume $T$ is renormalisable $n$ times. There exists a constant $M(T) >0$ such that the following holds. We use the following notation $\pi_{\mathcal{P}}\big( \mathcal{R}^n(T) \big) = (\varphi_1^n, \cdots, \varphi_d^n) \in \big( \mathrm{Diff}_+^2([0,1]) \big)^d$. Then we have for all $i\leq d$ and for all $n \in \mathbb{N}$

$$ || \mathcal{R}^n(T)'' || \leq M' ||(T^{-1})'|| \cdot  || T'' || $$  

\end{proposition}

\begin{proof}

Let $(\varphi_1^n, \cdots, \varphi_d^n) $ be the renormalised branches of $T$. The proof is an application of Lemma \ref{secondderivative} to the composition of restrictions of $T$ to the dynamical partition. By definition  $\varphi_i^n$ is the renormalised of $T^{l^i_n}$ restricted to an interval $I_n^i$ such that $I^j_n, T(I^j_n),  T^2(I^j_n), \cdots, T^{l^j_n-1}(I^j_n)$ are disjoint. We denote by $S_k$ the restriction of $T$ to $T^k(I^j_n)$. We have the following properties 

\begin{itemize}

\item $\varphi_i^n =\mathrm{N}(S_{l^j_n-1})  \circ \cdots \circ \mathrm{N}(S_1)  \circ \mathrm{N}(S_0) $ 

\item any partial product $\psi_k = \mathrm{N}(S_{l^j_n-1})  \circ \cdots \circ \mathrm{N}(S_k)$ is such that $||\log (\psi_k)'|| \leq K ||T''||$ ($\psi_k$ is a diffeomorphism of $[0,1]$ and therefore there exists $x_0 \in [0,1]$ such that $\log \psi'_k (x_0) = 0$ and the claim follows from Lemma \ref{bound1});

\item same holds for partial products $\phi_k = \mathrm{N}(S_{k})  \circ \cdots \circ \mathrm{N}(S_0)$;

\item for any $k$, $||\mathrm{N}(S_k)||' \leq ||(T^{-1})'|| \cdot ||T''|| \cdot |T^k(I^j_n)| $.

\end{itemize}

\noindent The result is a consequence of Lemma \ref{secondderivative} applied to $\mathrm{N}(S_{l^j_n-1})  \circ \cdots \circ \mathrm{N}(S_1)  \circ \mathrm{N}(S_0)$. Indeed

$$  || \varphi_i^n)'' || \leq  \sum_{k=1}^n{  ||\phi_{n-k}'|| \cdot || \mathrm{N}''(S_{n-k+1})|| \cdot || \psi'_{n-k+2} ||  } $$ and replacing in the inequality 

$$ || \varphi_i^n)'' || \leq e^{2K||T''||} \cdot ||(T^{-1})'|| \cdot ||T''|| \sum_{k=0}^{n-1}{|T^k(I^j_n)|}.$$  The $T^k(I^j_n)$s are all disjoint and the $\exp$ being bounded on bounded sets, we get the uniform bound for the $(\varphi_i^n)''$. But because of the Denjoy-Koksma inequality (Theorem \ref{DK}) applied to $\log T'$, we get that $|\log (\mathcal{R}^i(T))'| \leq \mathrm{Var}(\log T')$, the ratio between $(\varphi_i^n)'$  and the derivative of the corresponding branch of $\mathcal{R}^i(T)$ is bounded by a uniform constant depending only on $T$. This implies the Proposition.

\end{proof}

\subsection{Convergence to $\mathbf{P}$}

\noindent We give a sketch of the proof as it is already covered in many places in the literature, see \cite{KhaninTeplinsky} for instance. The fast decay of the size of the dynamical partition implies that branches of $\mathcal{R}^nT$ look more and more like Moebius maps. This can be quantified using material borrowed from \cite{KhaninTeplinsky}. In this article, the authors introduce what they call the distortion of a diffeomorphism $f$ of the interval which encodes how cross-ratios are modified under the action of $f$. This distortion behaves nicely under compositions and it is easy to show using Lemma 6 in \cite{KhaninTeplinsky} that the $\log$ of the distortion of (each branch of) $\mathcal{R}^n$ is proportional to $\Delta_n$. The distortion of a map is close to $1$ if and only if it $\mathcal{C}^0$-close to a Moëbius map. But because of the $\mathcal{C}^2$-bounds,  $\mathcal{C}^0$-closeness implies $\mathcal{C}^1$-closeness. From this discussion we get Theorem \ref{convergenceMoebius}.

\section{Character varieties and symplectic structure}

In this section we show how connected components of $\mathbf{P}$ naturally identify with open sets of the space of representations of the fundamental group of a punctured torus into $\mathrm{PSL}(2,\mathbb{R})$. We then show that consecutive renormalisation are attracted by a codimension $1$ submanifold of $\mathbf{P}$ and that this submanifold identifies with a \textit{character variety}. This allows to endow the attractor of renormalisation with a natural \textit{symplectic structure}, thanks to the structural understanding of character varieties.

\noindent \textbf{In the sequel we make the standing assumption that all the $c_i$'s are different from $1$}.

\subsection{Associating a representation to a projective IET}

Fix a circular permutation $\sigma$ once and for all. Let $\mathbf{P} =\mathbf{P}_{\sigma}$ be the set of piecewise projective IETs with permutation $\sigma$. From now on, we allow PIETs to be maps from  \textit{any} interval $[a,b]$ to itself.

\vspace{3mm}

\paragraph*{\bf A polygonal model for the $n$-punctured torus}  For any such permutation $\sigma$, we associate a presentation of the $n$-punctured torus: we consider a $2(n+1)$-gon and we label its sides in this order $a_1,a_2, \cdots, a_{n+1}, a_{\sigma(n+1)}, a_{\sigma(n)}, \cdots, a_{\sigma(1)}$ and (topologically) glue together sides with the same label, as in the picture below:

\begin{figure}[!h]
  \centering
  \includegraphics[scale=0.4]{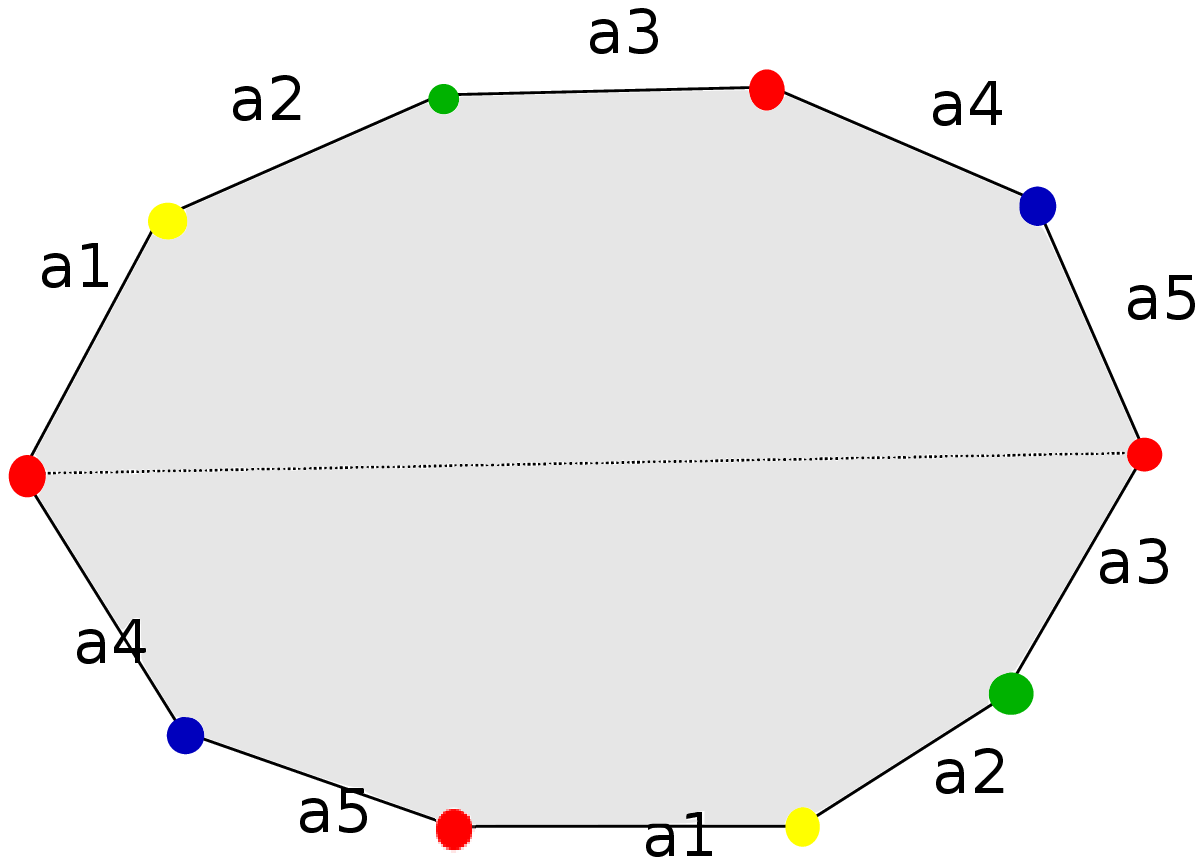}
  \caption{Presentation of a $5$-punctured torus.}
  \label{disco2}
\end{figure}

\noindent The boundary can be split into two broken lines, on at the top with  consecutive sides labelled  $a_1,a_2, \cdots, a_n$ and one at the bottom with consecutive sides labelled $ a_{\sigma(n)}, a_{\sigma(n-1)}, \cdots, a_{\sigma(1)}$.

\noindent These two lines are two be thought of as a geometric representation of the permutation $\sigma$, as it is classical to do in the world of interval exchange transformations.

\vspace{3mm}

\paragraph*{\bf The fundamental group of $\Sigma_{1,n}$.} Fix a point $p$ inside the polygon. For all $i$, consider a loop $\gamma_i$ based at p joining the middle points of the two sides labelled $a_i$ (one on the top line and one on the bottom line). We get this way $n+1$ closed loops $(\gamma_1, \cdots, \gamma_n)$  based at $p$.  See picture below 

\begin{figure}[!h]
  \centering
  \includegraphics[scale=0.4]{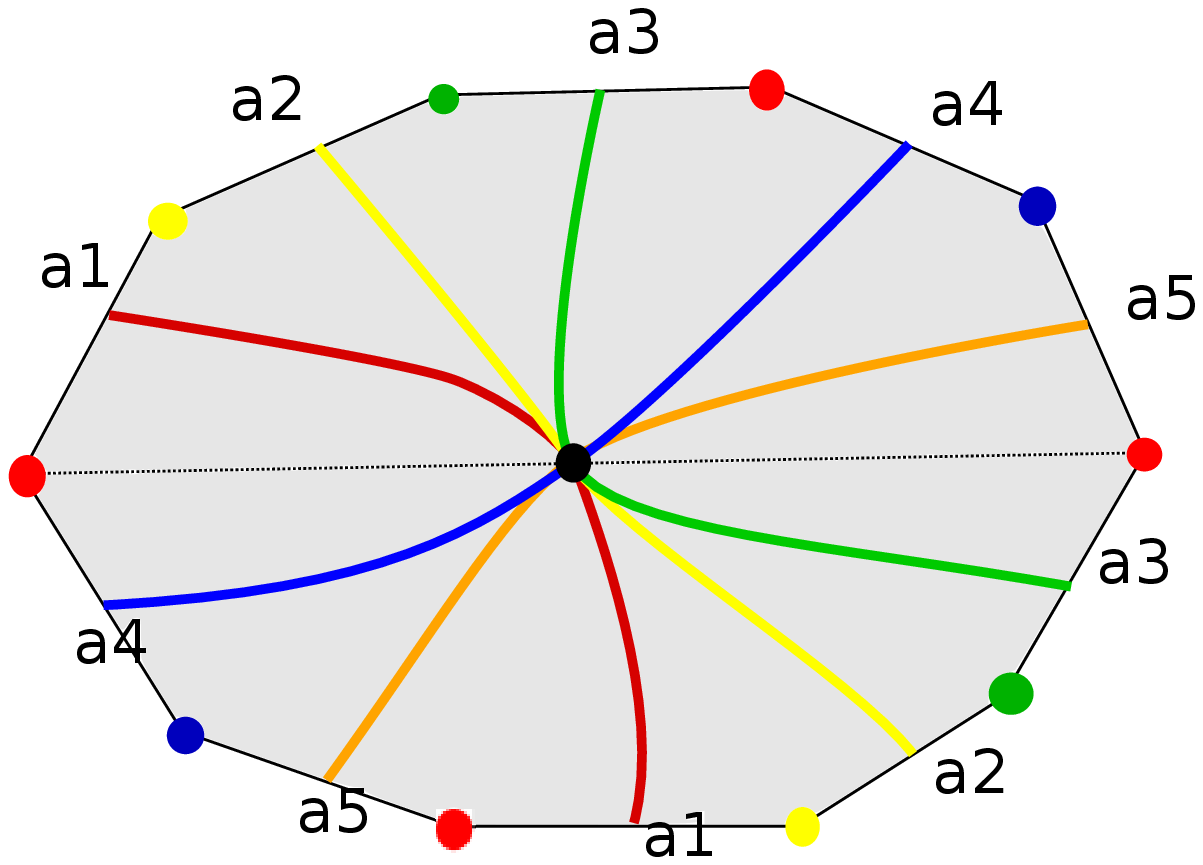}
  \caption{Presentation of a $5$-punctured torus.}
  \label{disco2}
\end{figure}

\begin{proposition}

The classes $[\gamma_1], \cdots, [\gamma_{n+1}]$ freely generate the fundamental group of $\Sigma_{n+1}$.

\end{proposition}

\begin{proof}
This is direct application of Van Kampen theorem. Another way to see it is to remark that $\Sigma$ deformation retracts onto the union $\bigcup{\gamma_i}$ whose fundamental group is the free group generated by the classes $[\gamma_i]$s.

\end{proof}

\vspace{3mm}

\paragraph*{\bf The representation associated to a PIET} Now consider $f$ an element of $\mathbf{P}_{\sigma}$. Let $I_1, \cdots, I_{n+1}$ its consecutive intervals of continuity. By definition, there exists an element $f_i \in \mathrm{PSL(2,\mathbb{R}})$ such that $f_{|I_i} = f_i$. 

\vspace{2mm}

\noindent Considering the above presentation of $\pi_1 \Sigma_{1,n}$, we define for all $f$ a representation $\rho : \pi_1 \Sigma_{1,n} \longrightarrow \mathrm{PSL}(2, \mathbb{R})$ by the following pairing 

$$ \rho: \gamma_i \longmapsto f_i .$$

\noindent Since $\pi_1 \Sigma_{1,n}$ is free and freely generated by the $\gamma_i$s, this defines a representation. We have defined a map 

$$ \psi : \mathbf{P}_{\sigma} \longrightarrow \mathrm{Hom} \big( \pi_1 \Sigma_{1,n},  \mathrm{PSL}(2, \mathbb{R}) \big).$$ Because $\pi_1 \Sigma_{1,n}$ is a free group of rank $n+1$, $\mathrm{Hom} \big( \pi_1 \Sigma_{1,n},  \mathrm{PSL}(2, \mathbb{R}) \big)$ identifies with $\mathrm{PSL}(2, \mathbb{R})^{n+1}$ and can therefore be endowed with a natural structure of smooth manifold. We now prove:

\begin{proposition}
\label{identification}
$ \psi : \mathcal{M}_{\sigma} \longrightarrow \mathrm{Hom} \big( \pi_1 \Sigma_{1,n},  \mathrm{PSL}(2, \mathbb{R}) \big) $ is a (local) diffeomorphism onto an open set of $\mathrm{Hom} \big( \pi_1 \Sigma_{1,n},  \mathrm{PSL}(2, \mathbb{R}) \big) $.
\end{proposition}

\begin{proof}

The proof rests on the following observation: the discontinuity points of the derivative (or say the ends of the intervals $I_1, \cdots, I_n$) are given as fixed points of the image by the associated representation of certain (fixed) elements of the group $\pi_1 \Sigma_{1,n}$. For instance, the right end point of the interval $I_1$ is one of the two fixed points of $\rho(\gamma_1\gamma_2^{-1})$. The element $\gamma_1\gamma_2^{-1})$ represents a simple closed curve going around the singularity in the torus corresponding to the first discontinuity of the derivative.
\noindent One can therefore locally extract the break points of a given PIET from its associated representation. Assume two PIETs have the same representation. Then by the above reasoning the have same break points and this breaks points have same image. By definition, the associated representation is defined by associating to certain curves the elements in $\mathrm{PSL}(2,\mathbb{R})$ which define the PIET when restricted to its intervals of smoothness. Therefore the datum of break points and the representation is enough to characterise a given PIET. The break points being themselves characterised by the representation, we get that the map $\psi$ is locally injective. 

\noindent To deduce local surjectivity, it is enough to reconstruct the PIET associated to a representation $\rho$ by using the inverse operation described above. Identify the position of the break points as they appear as fixed points of the image of certain elements $\pi_1\Sigma_{1,n}$ and define the PIET which is the restriction of $\rho(\gamma_1), \cdots, \rho(\gamma_{n+1})$ to the intervals defined by the breaks points.

\end{proof}

\begin{remark}
The above isomorphism is interesting but contains some useless information. Indeed, every PIET can be conjugated by an affine map to a map whose domain is $[0,1]$. This correspond to conjugating $\mathrm{Hom} \big( \pi_1 \Sigma_{1,n},  \mathrm{PSL}(2, \mathbb{R}) \big) $ by the subgroup of $\mathrm{PSL}(2, \mathbb{R})$ of affine maps, \textit{i.e.} the subgroup of matrices of the form $$\{ \begin{pmatrix}
\lambda &  t \\
0 & \lambda^{-1}
\end{pmatrix} \ | \lambda >1 \ \text{and} \ t \in \mathbb{R} \}.$$ 

\end{remark}

\vspace{3mm}

\paragraph{\bf Renormalisation from the representation viewpoint}
\label{renormalisation}

Geometrically, each step of renormalisation corresponds to a 'cut and paste' operation at the level of polygonal models. This cut and paste is coherent with the representation point of view, in  the following sense: 

\begin{enumerate}

\item a cut and paste realises a diffeomorphism $$ \phi : \Sigma_{1,n} \longrightarrow \Sigma_{1,n}$$

\item this diffeormorphism realises a group homomorphism 

$$  \phi_* : \pi_1 \Sigma_{1,n} \longrightarrow \pi_1 \Sigma_{1,n}$$

\item Renormalisation, read at the level of the parametrisation by $\mathrm{Hom} \big( \pi_1 \Sigma_{1,n},  \mathrm{PSL}(2, \mathbb{R}) \big) $ is equal to 

$$ \rho \longmapsto \rho \circ \phi_* $$

\end{enumerate}

\paragraph{\bf The attracting family}
\label{attracting}
We prove in this paragraph that there is a codimension $1$ family of $\mathbf{P}_{\sigma}$ which is invariant by renormalisation and which also eventually attracts everything. This family is analogous to that of \cite{KhaninKhmelev}.

\noindent This subfamily is defined by forcing a preferred representative in $\mathrm{Hom} \big( \pi_1 \Sigma_{1,n},  \mathrm{PSL}(2, \mathbb{R}) \big) $ for a certain loop going around a singularity. Precisely, the loop going around the 'distinguished' point crossing first the leftmost side at the 'top' of the polygon associated to a permutation.

\vspace{2mm} For every $\sigma$ permutation of genus $1$, we define the subfamily $\mathcal{E}_{\sigma} \subset \mathbf{P}_{\sigma}$ by the following: $ f \in \mathcal{E}_{\sigma}$ if and only if 

\begin{enumerate}

\item the associated PIET has domain $[0,1]$;

\item the loop around the 'distinguished' point is belongs to the affine group. 

\end{enumerate}

\begin{proposition}
\begin{enumerate}

\item The family $\bigcup_{\sigma}{\mathcal{E}_{\sigma}}$ is invariant by Rauzy induction.

\item Renormalisation converges to  $\bigcup_{\sigma}{\mathcal{E}_{\sigma}}$ exponentially fast(up to natural acceleration).

\end{enumerate}

\end{proposition}

\begin{proof}

By definition, renormalisation (before normalisation) does not change the image of the corresponding representations evaluated on the \textit{distinguished loop}. Rescaling by an affine element corresponds to conjugating by an affine element. It therefore

\begin{itemize}
\item preserves $\bigcup_{\sigma}{\mathcal{E}_{\sigma}}$;

\item makes everything converge to $\bigcup_{\sigma}{\mathcal{E}_{\sigma}}$.
\end{itemize}

Indeed, take any element in $f \in \mathrm{PSL}(2,\mathbb{R})$, and let $g_{\lambda}:= x \mapsto \lambda x$. It is easily shown that 

$$ g_{\lambda} \circ f \circ g_{\lambda}^{-1} $$ converges to the affine group  when $\lambda$ goes to $0$ at a speed proportional to $\lambda$. We can therefore deduce from this remark that consecutive renormalisations of $T$ a $\infty$-renormalisable PIET converge towards $\mathcal{E}_{\sigma}$ at the same speed as the first-return interval corresponding to $\mathcal{R}^nT$ converges to $0$. By estimates of Proposition \ref{estimatessize}, we get the exponential convergence.

\end{proof}

\vspace{3mm}
\paragraph{\bf Fixing the size of the break}

Each loop around a marked point of $\Sigma_{1,n}$ defines a conjugacy class in $\pi_1 \Sigma_{1,n}$. 

\noindent Consider the representation $\rho : \Sigma_{1,n} \longrightarrow \mathrm{PSL}(2,\mathbb{R})$  which represents a PIET. 

\begin{proposition}
\label{size}
Let $t_1$ a break point of $f$ and let $p_1$ the associated marked point of $\Sigma_{1,n}$ and let $\gamma_1$ a loop around that point. Then the size of the break at $p_1$ is exactly determined by the trace of $\rho(\gamma_1) \in \mathrm{PSL}(2,\mathbb{R})$.

\end{proposition}

\begin{proof}

Let $f \in \mathrm{PSL}(2,\mathbb{R})$ be a hyperbolic element and let $\lambda$ be $\mathrm{Tr}(f)$ (the trace of an element in $\mathrm{PSL}(2,\mathbb{R})$ is the absolute value of the trace of a representative in $\mathrm{SL}(2,\mathbb{R}))$. Let $x$ be a fixed point of $f$. Then we have the following formula

$$ \mathrm{Tr}(f) = (f'(x))^2 + (f'(x))^{-2}.$$ Indeed, any hyperbolic element in $\mathrm{PSL}(2,\mathbb{R})$ is conjugate to an element of the form $\begin{pmatrix}
\lambda & 0 \\
0 & \lambda^{-1}
\end{pmatrix}$ whose action on $\mathbb{R}$ is $t \mapsto \lambda^2t$ hence the result. But the size of the break at $p_1$ is exactly the derivative of $\gamma_1$ at one of its fixed points. Consequently, the size of the break at $t_1$ is(locally) entirely determined by the trace of $\rho(\gamma_1)$.

\end{proof}

\noindent By Proposition \ref{size}, fixing the size of all breaks, say $c = (c_1, \cdots, c_n)$ leads to a subfamily in both $\mathbf{P}_{\sigma}$ and $\mathcal{E}_{\sigma}$ that we denote by  $\mathbf{P}_{\sigma}^c$ and $\mathcal{E}_{\sigma}(\mathbf{c})$.

\subsection{Character varieties}

In the sequel, we introduce material about representations of surface groups into a reductive Lie group, and quickly specialise to the case where this reductive Lie group is $\mathrm{PSL}(2,\mathbb{R})$. Representations of surface groups are objects of interest in their own right, they appear naturally in many contexts such as Teichmüller theory, complex differential equations, the study of geometric structures on surfaces and non-abelian Hodge theory.  \noindent Given a Lie group $G$ and a topological surface $\Sigma$, one might be interested in the moduli space of representations of $\pi_1 \Sigma$ into $G$ which we naturally denote by  $\mathrm{Hom}(\pi_1 \Sigma, G)$. This moduli space carries a natural $G$ action by conjugation, and in many geometric contexts, it is the quotient of  $\mathrm{Hom}(\pi_1 \Sigma_{g,n}, G)$ by this action which is the relevant object to consider (for instance, the \textit{holonomy representation} of a geometric structure is well-defined up to conjugation and only defines an element in this quotient).

\begin{definition}[Character varieties]

Let $G$ be a reductive Lie group. The character variety of $\Sigma_{g,n}$ into $G$ is 

$$ \chi_{g,n}(G) $$ and defined to be the Hausdorff quotient of $\mathrm{Hom}(\pi_1 \Sigma_{g,n}, G)$ by the action of $G$ by conjugation. 

\end{definition}

\noindent Such character varieties are naturally endowed with an action of the \textit{mapping class group} of $\Sigma_{g,n}$ which is just the quotient of the action of the automorphism group of $\pi_1 \Sigma$ (see \cite{FM}). Suppose $n> 0$. Let $\gamma_1, \cdots, \gamma_n \in \pi_1 \Sigma_{g,n}$ simple closed loops around $p_1, \cdots, p_n$ the marked points of $\Sigma_{g,n}$. A free homotopy class of a loop in a manifold corresponds to a conjugacy class in its fundamental group. 

\begin{definition}[Relative character varieties]

Let $G$ be a reductive Lie group. Let $\mathbf{h} = (h_1, \cdots, h_n)$ conjugacy classes in $G$.  The \text{relative character variety} of $\Sigma_{g,n}$ into $G$ is 

$$ \chi_{g,n}(G, \mathbf{h}) $$ the subset of $ \chi_{g,n}(G) $ for which element in $\pi_1 \Sigma_{g,n}$ representing $\gamma_i$ is mapped to an element of the conjugacy class $h_i$.

\end{definition}

\noindent From now onwards we assume that $G = \mathrm{PSL}(2,\mathbb{R})$. 

\begin{theorem}[\cite{GHJW}]
\label{GHJW}
Assume that for all $i$, $h_i$ is a non-trivial conjugacy class. Then the relative character variety $\chi_{g,n}(G, \mathbf{h})$ carries a symplectic form $\omega$ which is invariant by the action of the mapping class group of $\Sigma_{g,n}$.
\end{theorem}

\noindent We refer to \cite{GHJW} for a complete proof of this result (which is way beyond the scope of this article) and \cite{Goldman} for the case of closed surfaces which, although not being directly relevant to the situation at hand, provides a very nice exposition of the construction of the symplectic form in a simpler case.

\vspace{3mm}

\paragraph{\bf Identification with $\mathcal{E}_{\sigma}$}

We explain here the link between these character varieties and renormalisation. Note that conjucagy classes in $\mathrm{PSL}(2,\mathbb{R})$ correspond exactly to level sets of the trace function (apart from the pre-image of $\{1\}$).

$$ \mathrm{Tr} : \mathrm{PSL}(2, \mathbb{R}) \longrightarrow \mathbb{R}_+$$ In particular, fixing the size of the breaks of a PIET corresponds to fixing the conjugacy class of loops around marked points of the associated representation in $\mathrm{PSL}(2,\mathbb{R})$. Let $h = (h_1, \cdots, h_n)$ be the conjugacy classes in $\mathrm{PSL}(2, \mathbb{R})$ corresponding to the choice of breaks $c = (c_1, \cdots, c_n)$.

\begin{proposition}

The projection $\pi : \mathcal{E}_{\sigma}(\mathbf{c}) \longrightarrow \chi_{1,n}(\mathrm{PSL}(2,\mathbb{R}), \mathbf{h} ) $ is a local diffeomorphism.

\end{proposition}

\begin{proof}

Consider $T$ an element of $\mathbf{P}_{\sigma}$ and $\rho$ its associated representation. Conjugating $\rho$ by an element $g \in \mathrm{PSL}(2,\mathbb{R})$ is equivalent to conjugating $T$ by $g$. The claim above is equivalent to the existence of a unique element of $\mathcal{E}_{\sigma}(\mathbf{c})$ in each $\mathrm{PSL}(2,\mathbb{R})$-conjugacy classes. This is shown to be true the following way: any hyperbolic element of $\mathrm{PSL}(2,\mathbb{R})$ is conjugate to an affine map. We can therefore conjugate $T$ in order to make the distinguished loop (the one going around the distinguished point) affine. The resulting PIET is a map form an interval $[a,b]$ to itself. From this point it suffices to conjugate by an affine map mapping $[a,b]$ onto $[0,1]$. The uniqueness is guaranteed by the following fact: an orientation preserving affine map which maps $[0,1]$ to itself is the identity.

\end{proof}

\noindent We deduce from that the following corollary 

\begin{corollary}

The pull-back of the natural symplectic form defines on $\bigcup_{\sigma}{E_{\sigma}(\mathbf{c}}$ a symplectic form $\omega$ locally invariant by the renormalisation operator $\mathcal{R}$.

\end{corollary}

\begin{proof}

Combination of Theorem \ref{GHJW} and Subsection \ref{renormalisation}.

\end{proof}

\paragraph{\bf Dimension counts}

We give here the dimension of character varieties for $g=1$. 

$$ \dim \chi_{1,d}(\mathrm{PSL}(2,\mathbb{R}) ) =  3d$$

$$ \dim \chi_{1,d}(\mathrm{PSL}(2,\mathbb{R}), \mathbf{h} ) =  2d $$

\noindent In particular, the dimension of $\mathcal{E}_{\sigma}$ is $3d$ and for any $\mathbf{c}$ the dimension of $\mathcal{E}_{\sigma}(\mathbf{c})$ is $2d$

\section{Comments and open problems}

\subsection{Dynamics of the renormalisation operator}

One important objective of renormalisation theory in the context of circle diffeomorphisms is to characterise $\mathcal{C}^1$-conjugacy classes. It has been conjectured for a long time that under mild arithmetic conditions, the decorated rotation number determines rigidity classes. 

\begin{conjecture}
\label{conjecture1}
There is a full-measure set of decorated rotation numbers $\gamma$ for which the following statement holds. Let $T_1$ and $T_2$ be circle diffeomorphisms with breaks, with breaks of the same size and same decorated rotation number. Then $T_1$ and $T_2$ are smoothly conjugate.
\end{conjecture}
 
\noindent This conjecture is implied by a sufficiently refined understanding of the dynamics of $\mathcal{R}$. Precisely we have that under arithmetic conditions (which have to do with the existence of "small divisors" type of problem), exponentially fast convergence of $\mathrm{d}_{\mathcal{C}^1}(\mathcal{R}^nT_1,\mathcal{R}^nT_2)$ to $0$ implies $\mathcal{C}^1$-conjugacy between $T_1$ and $T_2$. 

\vspace{2mm}

\noindent In \cite{KhaninYampolsky}, the authors show that in the case of one break point, restricted to the set of irrational rotation numbers, the renormalisation operator is hyperbolic and that level sets of the rotation number function are stable spaces of $\mathcal{R}$. This is enough to prove Conjecture \ref{conjecture1} in this case. This motivates the following conjecture

\begin{conjecture}
\label{conjecture2}
The renormalisation operator $\mathcal{R}$, restricted $\mathcal{E}_{\mathbf{c}}$, is hyperbolic restricted to its attractor.
\end{conjecture}

\noindent Recall that the dimension of $\mathcal{E}_{\mathbf{c}}$ is equal to $2d$. Heuristically, the decorated rotation number accounts for $d$ unstable directions. The existence of these unstable directions is known in the linear case, it is the hyperbolicity of the Rauzy-Veech induction restricted to \textit{standard interval exchange transformations}.  The existence of an invariant symplectic form would allow us to immediately construct the stable directions. Indeed, it is a standard property of symplectic matrices that an eigenvalue $\lambda \neq 1$ comes with the eigenvalue $\lambda^{-1}$ with equal multiplicity.

\subsection{Connected components of character varieties}

\noindent Understanding the topology of character varieties is a challenging problem in itself, even more so in the case of surfaces with boundary components. Classification of connected components of $\mathrm{PSL}(2,\mathbb{R})$-relative character varieties is worked out in Mondello's article \cite{Mondello}.

To a representation $\rho : \pi_1 \Sigma_{g,n} \longrightarrow \mathrm{PSL}(2, \mathbb{R})$ one can associate an algebraic invariant called its \textit{Euler number}, which we denote by $\mathrm{eu}(\rho) \in \mathbb{R}$. In the case $n=0$ it is nothing but the characteristic class defined by the flat bundle of rank $2$ over $\Sigma_g$ of monodromy $\rho$ and is integer-valued. The definition in the case with boundary is slightly more involved as it requires to take care of extra terms coming from the boundary.

If $\mathbf{h}$ is fixed, it is well-known that $$ \mathrm{eu} : \chi_{g,n}(\mathrm{PSL}(2,\mathbb{R}), \mathbf{h}) \longrightarrow \mathbb{R}$$ takes only finitely many values and is constant on connected components. Modello shows in \cite{Mondello} using the theory of Higgs bundles, that this Euler number actually characterises connected components. We ask the following question

\begin{question}

To which connected components belong representations constructed from piecewise Moebius circle diffeomorphisms with breaks?

\end{question}
 
 \noindent There is at least one component that can be excluded. The component with maximal Euler number identifies with the Teichmüller space of complete hyperbolic metrics. It is easily shown that the mapping class group acts properly discontinuously on this component. But on the other hand, recurrence properties of the renormalisation operator(which identifies with a "sub-action" of the mapping class group) ensure that the mapping class group action on a component appearing in the study of circle diffeomorphisms with breaks is not proper and discontinuous. Another question we think is interesting is the following

 \begin{question}
What are the representations which can arise as the representation of a minimal circle diffeomorphism with break?
\end{question}

 \subsection{Links to mapping class group dynamics}

As mentioned in the previous paragraph, the renormalisation operator and the mapping class group action on $\chi_{g,n}(\mathrm{PSL}(2,\mathbb{R}), \mathbf{h}) $ are closely related. In general, we know of the existence of special components for which the latter action is properly discontinuous. However, for other components it is suspected that this action is going to be ergodic with respect to the symplectic volume. This conjecture remains widely open in the general case although it has been proven in the (closed) genus $2$ case by Marché and Wolff \cite{MarcheWolff}. Another notable piece of work in the $\mathrm{SL}(2, \mathbb{C})$-case is \cite{Goldman4}, where the ergodic properties of the mapping class group action is completely analysed in the case of the one-holed torus.

\noindent Although it is not so clear to the authors what Conjecture \ref{conjecture2} implies for the mapping class group action on relative character varieties, it would certainly provide some interesting insight into its topological dynamics. Indeed, it would imply that a small sub-action displays some relatively rich dynamical behaviour. This sub-action could be replicated in many different way by conjugating in $\mathrm{MCG}(\Sigma_{1,n})$, potentially offering a path to a proof of the existence of dense orbits.

\bibliographystyle{alpha}
\bibliography{biblio}
 
\end{document}